\documentclass{amsart}

\usepackage{amsmath}
\usepackage{amsxtra}
\usepackage{amscd}
\usepackage{amsthm}
\usepackage{amsfonts}
\usepackage{amssymb}
\usepackage{eucal}
\usepackage[all]{xy}
\usepackage{graphicx}
\usepackage{tikz-cd}
\usepackage{mathrsfs}
\usepackage{euler}
\usepackage{hyperref}
\usepackage{color}
\usepackage{longtable}
\usepackage{float}
\usepackage{caption}
\usepackage{thmtools}
\usepackage{thm-restate}
\usepackage{enumitem}
\usepackage{url}

\usepackage[colorinlistoftodos, textsize=tiny]{todonotes}
\def\listtodoname{List of Todos}
\def\listoftodos{\@starttoc{tdo}\listtodoname}

\RequirePackage{color}
\definecolor{myred}{rgb}{0.75,0,0}
\definecolor{mygreen}{rgb}{0,0.5,0}
\definecolor{myblue}{rgb}{0,0,0.65}

\usepackage{tikz}
\usepackage{tikz-cd}
\usetikzlibrary{matrix,arrows,decorations.pathmorphing}

\theoremstyle{plain}
  \newtheorem{thm}{Theorem}[subsection]
  \newtheorem{prop}[thm]{Proposition}
  \newtheorem{lem}[thm]{Lemma}
  \newtheorem{cor}[thm]{Corollary}
	
	\newtheorem{question}[thm]{Question}
	
\theoremstyle{definition}
  \newtheorem{defn}[thm]{Definition}
  \newtheorem{example}[thm]{Example}

  \newtheorem{convention}[thm]{Convention}

\theoremstyle{remark}
	\newtheorem{rem}[thm]{Remark}
  
  \newtheorem*{note*}{Notation}

\numberwithin{equation}{section}

\newcommand\nc{\newcommand}
\nc\on{\operatorname}
\nc\renc{\renewcommand}
\newcommand\ssec{\subsection}
\newcommand\sssec{\subsubsection}

\newcommand\bn{{\mathbb N}}

\newcommand\bq{{\mathbb Q}}
\newcommand\bp{{\mathbb P}}
\newcommand\bz{{\mathbb Z}}
\newcommand\ba{{\mathbb A}}
\newcommand\bk{{\Bbbk}}

\newcommand\sco{{\mathscr O}}

\DeclareMathOperator{\sym}{Sym}

\DeclareMathOperator\di{Div}
\newcommand\bida{a}
\newcommand\bidb{b}
\newcommand\pdeg{\delta}
\newcommand\hirz{F}
\newcommand\mss{\mathscr{S}}

\DeclareMathOperator{\fr}{frac}
\DeclareMathOperator{\supp}{Supp}

\DeclareMathOperator{\newspan}{span}
\DeclareMathOperator{\proj}{Proj}

\DeclareMathOperator{\Te}{T_=}
\DeclareMathOperator{\Tp}{T_+}
\DeclareMathOperator{\Tm}{T_-}
\DeclareMathOperator{\num}{p}
\DeclareMathOperator{\den}{q}
\DeclareMathOperator{\lcm}{lcm}
\DeclareMathOperator{\Pic}{Pic}
\makeatletter
\newcommand{\customlabel}[2]{%
   \protected@write \@auxout {}{\string \newlabel {#1}{{#2}{\thepage}{#2}{#1}{}} }%
   \hypertarget{#1}{#2}
}
\newcommand\invisiblesubsection{%
  \refstepcounter{subsection}%
}
\makeatother

\title[Section rings of $\bq$-divisors]{Section rings of $\bq$-divisors on minimal rational surfaces}

\author{Aaron Landesman}
\address[Aaron Landesman]{Department of Mathematics, Stanford University}
\email{aaronlandesman@stanford.edu}

\author{Peter Ruhm}
\address[Peter Ruhm]{Department of Mathematics, Stanford University}
\email{pruhm@stanford.edu}

\author{Robin Zhang}
\address[Robin Zhang]{Department of Mathematics, Columbia University}
\email{rzhang@math.columbia.edu}

\date{\today}

\begin{document}
\begin{abstract}
 	We bound the degrees of generators and relations of
	section rings associated to arbitrary $\bq$-divisors on projective spaces of 
	all dimensions and Hirzebruch surfaces. For section rings of
	effective $\bq$-divisors on projective spaces, we find 
	the best possible bound on
	the degrees of generators and relations.
\end{abstract}

\maketitle

\section{Introduction}
\label{sec:intro}
For any Weil $\bq$-divisor $D$ on a rational surface $X$, the graded
\textbf{section ring} is $R(X, D) := \bigoplus_{d \geq 0}
H^0(X, \lfloor dD \rfloor)$.
In the case that $D = K_X$, where $K_X$ is the canonical divisor,
the graded section ring is referred to as
``the canonical ring'' and is a classical object of study.
For example, if $C$ is a curve of genus $g \geq 4$, Petri's theorem
relates the geometry of the curve $C$ to the canonical ring:
$R(C, K_C)$ is generated in degree 1 with relations in degree 2 
unless $C$ is hyperelliptic, trigonal, or a plane quintic
(see
\cite[p. 157]{saint-donat:proj} and \cite[Section 3.3]{acgh:algebraic-curves}).
In this way, explicit descriptions of generators and relations of section rings yield
geometric information about the underlying variety.

One natural way to generalize the classical result of Petri
mentioned above is to examine the section rings of stacky curves
(i.e., smooth proper geometrically connected 1-dimensional
Deligne-Mumford stacks over a field with a dense open subscheme).
These were studied by Voight--Zureick-Brown 
\cite{vzb:stacky} and Landesman--Ruhm--Zhang \cite{lrz:spin-cring},
which provide tight bounds on the degree of generators and relations
of log canonical rings and log spin canonical rings on arbitrary
stacky curves. All rings of modular forms associated to Fuchsian
groups can be realized as canonical rings of
such curves, so the above work also yields insight into such rings
of modular forms.
Further, 
O'Dorney \cite{dorney:canonical} gives similar descriptions
of section rings for arbitrary $\bq$-divisors on $\bp^1$
(as opposed to just log canonical $\bq$-divisors and log spin canonical $\bq$-divisors).

Beyond section rings of curves, section rings of certain higher dimensional
stacks have also been studied.
For example, the Hassett-Keel program \cite{hassett:classical-and-minimal-models}
studies log canonical rings on ${\mathscr M}_g$ of the form 
\begin{align*}
	\overline {\mathscr M}_g(\alpha) := \bigoplus_{d \geq 0} H^0 \left(
	\overline {\mathscr M}_g, \lfloor d K_{\overline{\mathscr M}_g} +
	\alpha\delta \rfloor  \right) 
\end{align*}
in terms of certain moduli spaces, where $\overline {\mathscr M_g}$ is the
moduli space of stable genus $g$ curves.

Moreover, $\bq$-divisors on surfaces not only appear in the context of stacks,
but also naturally
appear when considering the canonical ring of surfaces of Kodaira dimension 1.
In characteristic 0, every surface of Kodaira dimension 1 is an elliptic surface \cite[p. 244]{
barthHPV:compactComplexSurfaces}, and
the canonical rings of elliptic surfaces are most naturally described in the setting of
$\bq$-divisors \cite[Chapter V, Theorem 12.1]{barthHPV:compactComplexSurfaces}. 

In this paper, we continue the study of section rings of
surfaces and higher dimensional varieties, examining section rings of
projective spaces $\bp^m$ and Hirzebruch surfaces $\hirz_m$.
When $D$ is a general $\bq$-divisor on $\bp^m$ or
$\hirz_m$, we give bounds on the generators and relations of
$R(X, D)$. In particular, we give a presentation of the
section ring when $D$ is any effective
$\bq$-divisor on $\bp^m$.  

Throughout, for ease of notation,
we work over a fixed algebraically closed field $\bk$.
This is nonessential, see \autoref{rem:algebraically-closed}

\ssec{Main Results and Outline}
\label{ssec:main}

After briefly stating our notation in Section~\ref{section:notation},
in Section~\ref{sec:proj}, we prove the following two results bounding
the degree of generators and relations of
section rings on $\bp^m$. The first applies
to effective $\bq$-divisors and the second applies to arbitrary
$\bq$-divisors. Note that here we allow coefficients of divisors to be 0,
see Section~\ref{ssec:arbitrary-projective}.

\begin{restatable}{thm}{restateeffective}
\label{thm:proj-effective-intro}
Let $D = \sum_{i = 0}^{n} \alpha_i D_i \in \di \bp^m \otimes_\bz \bq$, with $\alpha_i =
\frac{c_i}{k_i} \in \bq_{> 0}$ in reduced form and each $D_i \in \di \bp^m$ an integral divisor.

Then the section ring $R(\bp^m, D)$
is generated in degrees at most $\max_{0 \leq i \leq n}{k_i}$ with
relations generated in degrees at most $2 \max_{0 \leq i \leq n}{k_i}$.
\end{restatable}

\begin{restatable}{thm}{restateproj}
\label{thm:proj-generators-relations}
Let $D = \sum_{i = 0}^n \alpha_i D_i \in \di \bp^m \otimes_\bz \bq$
with $\alpha_i = \frac{c_i}{k_i}\in \bq$ in reduced form. Write
$\ell_i := \lcm_{0\le j \le n, j \ne i}(k_j)$ and $a_i := \deg D_i$. Let $\bp^m
\cong \proj \bk[x_0, \ldots, x_m]$ and let
$f_i \in \bk[x_0, \ldots, x_m]$ such that $D_i = V(f_i)$. 
Suppose that $\{f_0, \ldots, f_n\}$ contains a basis
for $H^0(\bp^m, \sco_{\bp^m}(1))$, {\rm(}i.e. $m+1$ independent polynomials with
corresponding $a_i = 1${\rm)}.

Then $R(\bp^m, D)$ is generated in degrees at most 
$\omega := \sum_{i = 0}^n \ell_i a_i$
with relations generated in degrees at most
\[
	\max \left(2 \omega, \frac{\max_{0 \le i \le n}
	(	\bida_i)}{\deg(D)} + \omega \right).
\]
\end{restatable}

\begin{rem}
\label{rem:proj-tight-bounds}
The bounds
given in Theorem ~\ref{thm:proj-effective-intro}
are tight and are typically attained, as explained in
Remark ~\ref{rem:exact-effective-bounds}.
Similarly, the bounds given in Theorem 
~\ref{thm:proj-generators-relations} are asymptotically
tight to within a factor of two 
for a class of divisors described in
Remark~\ref{rem:exact-noneff-bounds}.
\end{rem}

Note that the assumption given in Theorem~\ref{thm:proj-generators-relations}
that $\left\{ f_0, \ldots, f_n \right\}$ contains a basis for $H^0(\bp^m, \sco_{\bp^m}(1))$ is not a serious assumption:
If $D$ is an arbitrary divisor, then $D + 0 \cdot V(x_0) + \cdots + 0 \cdot V(x_m)$ is a divisor satisfying this assumption. Plugging this divisor into the statement
of Theorem~\ref{thm:proj-generators-relations}, we obtain the following bounds for an arbitrary divisor:
\begin{cor}
\label{cor:proj-generators-relations}
Let $D = \sum_{i = 0}^n \alpha_i D_i \in \di \bp^m \otimes_\bz \bq$ 
with $\alpha_i = \frac{c_i}{k_i}\in \bq$ in reduced form. Write
$\ell := \lcm_{0 \leq j \leq n}(k_j)$, $\ell_i := \lcm_{0\le j \le n, j \ne i}(k_j)$ and $a_i := \deg D_i$. Let $\bp^m
\cong \proj \bk[x_0, \ldots, x_m]$ and let
$f_i \in \bk[x_0, \ldots, x_m]$ such that $D_i = V(f_i)$. 

Then $R(\bp^m, D)$ is generated in degrees at most 
$\omega' := (m+1)\ell + \sum_{i = 0}^n \ell_i a_i$
with relations generated in degrees at most
\[
	\max \left(2 \omega', \frac{\max_{0 \le i \le n}
	(	\bida_i)}{\deg(D)} + \omega' \right).
\]
\end{cor}

In Section \ref{sec:hirz}, we shift our attention to Hirzebruch surfaces. Recall that for each $m \geq 0$ we define the
Hirzebruch surface $F_m := \proj \sym (\sco_{\bp^1} \oplus
\sco_{\bp^1}(m))$, viewed as a projective bundle over $\bp^1$.
Let $u,v$ the projective coordinates on the base $\bp^1$ and let $z,w$ be the projective coordinates
on the fiber, as defined more precisely at the beginning of
\ref{sec:hirz}.

\begin{thm}
\label{thm:hirz-generators-relations}
Let $D = \sum_{i = 1}^n \alpha_i D_i \in \di \hirz_m \otimes_\bz \bq$ where $\alpha_i = \frac{c_i}{k_i} \in \bq$ is written in reduced form.  
Let each $D_i = V(f_i)$, where $f_i \in \sco(a_i, b_i)$.
Let $u, v, z, w$ be the coordinates for the Hirzebruch surface
$\hirz_m$, and suppose that $\{f_1, \ldots, f_n\}$
contains bases for $\sco_{F_m}(1,0)$ and $\sco_{F_m}(0,1)$
{\rm(}i.e., two independent linear polynomials in $u, v$ and two independent linear polynomials in
$w, z$ {\rm)}.

Then $R(\hirz_m, D)$ is generated in degrees at most
\[
	\rho' := \lcm_{1\le i\le n}(k_i)\cdot \left(\sum_{1\le i\le j\le n} \bida_i \bidb_i\right) 
\]
\noindent
with relations generated in degrees at most 
$2 \rho'$.
\end{thm}

\begin{rem}
Theorem~\ref{thm:hirz-generators-relations} is restated with a more
precise bound in Theorem~\ref{thm:hirz-generators-relations-full}.
\end{rem}

\begin{rem}
	\label{rem:hirz-generators-relations-hypothesis}
	As in Theorem~\ref{thm:proj-generators-relations}, the condition
	of Theorem~\ref{thm:hirz-generators-relations} that $D$ contain
	independent polynomials in $u, v$ and $w, z$ is
	easily removed by replacing $D$ with $D' :=D +0 \cdot u + 0 \cdot
	v + 0 \cdot w + 0 \cdot z$,	and computing the resulting bound for
	$D'$ using either the simpler (but slightly weaker) bound
	from Theorem~\ref{thm:hirz-generators-relations} or the more
	precise bound given in Theorem~\ref{thm:hirz-generators-relations-full}.
\end{rem}

\noindent
By the classification of minimal rational surfaces
as either $\bp^2$ or a Hirzebruch surface \cite{eisenbud-harris:minimal},
Theorems \ref{thm:proj-generators-relations} and
\ref{thm:hirz-generators-relations} provide bounds on generators
and relations for the section ring of any $\bq$-divisor on any
minimal rational surface. Section \ref{sec:conc} discusses further questions.

\subsection{Modular Forms}
\label{subsection:modular-forms}
The work in this paper was motivated by 
potential applications to
calculating a presentation of certain rings of
Hilbert 
modular forms and Siegel modular forms. 
Recall that Hilbert and Siegel 
modular forms are two generalizations of modular forms to higher dimensions,
as described in \cite{geer:siegel-modular} and \cite{bruinier:hilbert-modular}.
If the bounds given in 
Theorem~\ref{thm:proj-generators-relations} and Theorem~\ref{thm:hirz-generators-relations}
could be extended to all rational surfaces, instead of just minimal ones,
they would give a bound on generators and relations for 
rings of Hilbert and Siegel modular forms parametrized by rational
Hilbert and Siegel modular surfaces.
See Section~\ref{sec:conc} for a potential approach for generalizing our results to arbitrary rational surfaces.
This would be interesting because such modular surfaces tend to be
immensely complicated. 

Since our results only apply to minimal rational surfaces, and because the rings of Hilbert and Siegel modular forms are so complex, we were unable
to use our work to compute the section ring of an explicit modular surface.
Although we did not obtain a bound on the degree of generators and relations
for general rational surfaces, the restricted class of rational varieties we consider still required significant work.

\section{Notation}
\label{section:notation}
\invisiblesubsection

In this section, we now collect various notation used throughout the paper.
Throughout,
we work over a fixed algebraically closed field $\bk$ for ease of notation,
but our results hold equally well over arbitrary fields.

\begin{rem}
\label{rem:algebraically-closed}
Because cohomology commutes with flat base change (in particular
with field extensions) the dimensions of the graded pieces of the
section ring will be preserved under base change from $\bk$ to
$\overline \bk$. Therefore generators and relations are preserved
under arbitrary base field extension, and so their minimal
degrees are preserved. Consequently, there is no harm in
assuming $\bk = \overline \bk$ for our proofs. The bounds we give
hold equally well over arbitrary fields.

Note that if $L/\bk$ is an inseparable extension
and $X$ is a scheme over $\bk$, then the base change of the
canonical divisor $(K_X)_L$ may be different than the canonical
divisor of the base change $K_{X_L}$.
Therefore, the canonical ring may not be preserved under base change
along inseparable extensions. Nonetheless, the results we give are
not affected because given a divisor on a scheme over some base field,
the structure of that particular section ring is unchanged upon
base change to the algebraic closure.
\end{rem}

Let $D$ be a (Weil) $\bq$-divisor on a rational surface $X$ of the form
\begin{align*}
	D = \sum_{i=1}^{n}\alpha_i D_i \in \di X \otimes_\bz \bq.
\end{align*}\noindent
where $n \in \bz_{\geq 0}$ indexes the number of irreducible divisors in 
the above expansion of $D$, $\alpha_i \in \bq$, and $D_i \in \di X$ is an integral
codimension 1 closed subscheme of $X$. When it is convenient to do
so, we shall sometimes start the indexing at $0$, so that $i$ runs
from $0$ to $n$. In the case $X = \bp^m$, define the degree of $D$ by $\deg D :=
\sum_{i=1}^{n}\alpha_i \cdot \deg D_i$. The floor of a $\bq$-divisor $D$ is the divisor
$\lfloor D \rfloor := \sum_{i = 1}^{n}
\lfloor \alpha_i \rfloor D_i$.

Let $R(X, D) := \bigoplus_{d \geq 0} H^0(X, \lfloor dD \rfloor)$
denote the section ring associated to the $\bq$-divisor $D$. 
We often alternatively write
$R(X, D) := \bigoplus_{d \geq 0} u^d H^0(X, \lfloor dD \rfloor)$, where $u$ is a dummy
variable to keep track of the degree.
When $X$ is understood from context, we use $R_D$ as notation for $R(X, D)$.

We use $m \in \bz_{\geq 0}$ to index the dimension of a given
projective space $\bp^m$ and the type of the Hirzebruch surface
$\hirz_m$. If $S$ is a graded ring, we denote the $d$th graded
component of $S$ by $S_d$. If $r$ is a rational number, we let
$\fr(r) := r - \lfloor r \rfloor$ denote the fractional part of $r$. 
If $D \in \di X \otimes_\bz \bq$ is an arbitrary divisor, we denote
$h^0(X, D) := \dim_\bk H^0(X,D)$.

\section{Section Rings of Projective Space}
\label{sec:proj}
Let $\bk$ be a field and let $\bp^m$ denote $m$-dimensional
projective space over $\bk$. In this section, we prove
Theorem~\ref{thm:proj-generators-relations}, which bounds the degrees
of generators and relations of the section ring of any $\bq$-divisor
on $X = \bp^m$ for all $m \geq 1$. We also prove
Theorem ~\ref{thm:proj-effective-intro} to give an explicit
description of the generators of the section ring $R_D$ when $D$ is
an effective divisor.

If $\deg D < 0$, the
section ring is concentrated in degree 0, and if
$\deg D = 0$, then the section ring has a single
generator. Therefore, for the remainder of this section,
we shall assume $\deg D > 0$.
Note that the $\bp^1$ case, in particular, restricts to the results
of \cite{dorney:canonical}.

For the remainder of this section, we shall fix $m \geq 1$ and
choose an
isomorphism $\bp^m \cong \proj \bk[x_0, \ldots, x_m]$.

\ssec{Effective Divisors on Projective Space}
\label{ssec:proj-one-point}

In this subsection, we restrict attention to the case of effective
fractional divisors $D \in \di \bp^m \otimes_\bz \bq$. We give an explicit presentation of the section ring when $D$ is an effective
divisor.

\begin{convention}
Let $\bp^m \cong \proj \bk[x_0, \ldots, x_m]$. Let
$\vec{v} = (v_0, \ldots, v_m) \in \bz^{m + 1}$.  Then write
\[
	x^{\vec{v}} := \prod_{i = 0}^{m} x_i^{v_i}.
\]
\end{convention}

\begin{defn}
\label{defn:vec-sum}
For $\vec{v} \in \bz^n$, denote $\deg \vec{v} := \sum_{i = 0}
^n v_i$.
For a given sequence of numbers $c_0, \ldots, c_r$, let 
\begin{align*}
	\mss_i := \left \{\vec{v} \in \bz_{\geq 0}^{m + 1} \; : \;
\deg \vec v = c_i \right\}.	
\end{align*}

\noindent
\end{defn}

Next, we define an ordering on these vectors, which will be used to
give a presentation for $R_D$.

\begin{defn}
\label{defn:vec-order}
Let $\vec{v}, \vec{w} \in \bz^{m+1}$. Let $i \in \{0,\ldots, m\}$
be the biggest index such that $v_i$ is nonzero
and $j \in \{0,\ldots, m\}$ be the smallest index such that $w_j$ is
nonzero. Define a partial ordering
on $\bz^{m+1}$ by $\vec{v} \prec \vec{w}$ if $i \leq j$.
\end{defn}

We are now ready to give an inductive method for computing
the generators and relations of $R_D$ in terms of $R_{D'}$ 
in the case that $D' = R_D + \alpha H$ for $H$ a hyperplane and
$\alpha$ positive.
The statement and proofs are natural generalizations of
\cite[Theorem 6]{dorney:canonical}.

\begin{thm}
\label{thm:proj-one-point}
Let $\bp^m \cong \proj \bk [x_0, \ldots, x_m].$ Let $D' \in \di \bp^m \otimes_\bz \bq$ and $D = D' + \alpha H$, with $\alpha =
\frac{p}{q} \in \bq_{>0}$, $H := V(x_k)$ a hyperplane of $\bp^m$,
and $H \notin \supp(D').$
Let
\[
	0 = \frac{c_0}{d_0} <
	\frac{c_1}{d_1} < \cdots < \frac{c_r}{d_r} = \frac{p}{q}
\]

\noindent
be the convergents of the Hirzebruch-Jung continued fraction of
$\alpha$ {\rm(}q.v. \cite[Section 2]{voight:cf} and
\cite[Section 3]{hirzebruch:cf}{\rm)}. Then, the section ring
\[
	R_D := \bigoplus_{d \geq 0} u^d H^0(\bp^m, \lfloor dD \rfloor)
\]

\noindent
has a presentation over $R_{D'}$ consisting of the $\sum_{i = 0}^{r}
{{m + c_i} \choose {c_i}}$ generators $F_i^{\vec{v}} := \frac{u^{d_i}
x^{\vec{v}}}{x_k^{c_i}}$ where $0 \leq i \leq r$ and $\vec{v} \in \bz_{\geq 0}^{m + 1}$
with $\deg \vec v = c_i$. Furthermore, the ideal of
relations $I$ is generated by the following two classes of elements.
\begin{enumerate}
	\item For each $(i, j)$ with $j \geq i + 2$ and each $\vec{v} \in \mss_i,
		\vec{w} \in \mss_j$, there is either a relation of the form
		\begin{align*}
			G_{i, j}^{\vec{v}, \vec{w}} := F_i^{\vec{v}} F_j^{\vec{w}}
			- \prod_{\vec{y} \in \mss_{h_{i, j}}} (F_{h_{i, j}}^{\vec{y}})
			^{g_{\vec{y}}} \in I,
		\end{align*}
		with $i < h_{i, j} < j$, or there is a relation of the form
		\begin{align*}
			G_{i, j}^{\vec{v}, \vec{w}} := F_i^{\vec{v}} F_j^{\vec{w}}
			- \left( \prod_{\vec{y} \in \mss_{h_{i, j}}} (F_{h_{i, j}}^{\vec{y}})^{g_{\vec{y}}}
			\right) \left( \prod_{ \vec{z} \in
			\mss_{h_{i, j} + 1}} (F_{h_{i, j} + 1}^{\vec{z}})^{g'_{\vec{z}}} \right) \in I,
		\end{align*}
		with $i < h_{i, j} < h_{i, j} + 1 < j$.
	\item For each $(i, j)$ with
		$j = i$ or $j = i + 1$ and each $\vec{v} \in \mss_i, \vec{w} \in
		\mss_j$ with $\vec{v} \not\prec \vec{w}$ (see Definition
		~\ref{defn:vec-order}) there is a relation of the form
		\begin{align*}
			&L_{i, j}^{\vec{v}, \vec{w}} := F_i^{\vec{v}} F_j^{\vec{w}}
			- F_i^{\vec{y}} F_j^{\vec{z}} \in I
		\end{align*}

\end{enumerate}

\noindent
where $\vec{y}$ and $\vec{z}$ are the unique
vectors in $\mss_i$ and $\mss_j$, respectively, such that $\vec{y}
+ \vec{z} = \vec{v} + \vec{w}$ and $\vec{y} \prec \vec{z}$.
\end{thm}

\sssec*{Idea of Proof:}
The proof follows in three steps. 
First, since the
$F_{i}^{\vec{v}}$ generate
all of $R_D$ over $R_{D'}$, and the leading terms of
$G_{i,j}^{\vec v, \vec w}$ lies in $R_{D'}$, we obtain the
relations $G_{i, j}^{\vec{v}, \vec{w}}$. 
Then, we derive the relations
$L_{i, j}^{\vec{v}, \vec{w}}$ by considering when products of
generators in neighboring degrees are equal. Finally, we demonstrate
that $G_{i, j}^{\vec{v}, \vec{w}}$'s and $L_{i, j}^{\vec{v}, \vec{w}}$'s
generate all of the relations by using them to reduce arbitrary
elements of
$R_D$ to a canonical form.

\begin{proof}
As a first step, we reduce to the case $D' = 0$.
Using that $D$ is effective, we can write
$R_D = R_{D'} + R_{\alpha H}$ (where the sum is nearly a direct sum, except 
the $(R_{D'})_d \cap (R_{\alpha H})_d$ is the one dimensional subspace generated by $u^d$).
Then, to give a presentation of $R_D$ over $R_D'$,
it suffices to give a presentation of $R_{\alpha H}$:
generators of $R_{\alpha H}$ map to generators of $R_D$ over $R_{D'}$ under the inclusion $\iota: R_{\alpha H} \rightarrow R_D$
and relations map to a full set of relations for $R_D$ over $R_{D'}$.
Hence, for the remainder of the proof, we can assume $D' = 0$.

%
%

Next we show that $F^{\vec v}_i$ generate all of $R_D$. In the case $m = 1$, O'Dorney \cite[Theorem 6]
{dorney:canonical} demonstrates that each lattice point $(\beta, \gamma) \in
\bz_{\geq 0}^2$ with $\gamma \leq \beta \alpha$ lies in the $\bz_{\geq 0}$ span
of $(d_h, c_h)$ and $(d_{h + 1}, c_{h + 1})$ for
some $h \in \{0, \ldots, r\}$. A similar strategy works in the case $m > 1$. Let $(\beta, \gamma) = \lambda
(d_h, c_h) + \kappa (d_{h + 1}, c_{h + 1})$ for $\lambda, \kappa \in
\bz_{\geq 0}$. Any element $\frac{u^{\beta}
x^{\vec{v}}} {x_k^{ \gamma}} \in R_D$ is expressible as
\begin{align*}
	\frac{u^{\beta} x^{\vec{v}}} {x_k^{\gamma}} = \left(\frac{u^{d_h}}
	{x_k^{c_h}}\right)^{\lambda} \left(\frac{u^{d_{h + 1}}}
	{x_k^{c_{h + 1}}}\right)^{\kappa} x^{\vec{v}}.
\end{align*}

\noindent
We can then write $\vec{v}  = \sum_{\tau=1}^{\lambda}\vec{w}_{(\tau)} +
\sum_{\eta=1}^{\kappa} \vec z_{(\eta)}$ with $\vec w_{(\tau)} \in \mss_h$ and
$\vec z_{(\eta)} \in \mss_{h+1}$ to give a decomposition
\begin{align}
\label{eqn:one-point-canonical-form}
	\frac{u^{\beta} x^{\vec{v}}} {x_k^{\gamma}}	= \prod_{\tau = 1}
	^{\lambda} \frac{u^{d_h} x^{\vec{w}_{(\tau)}}} {x_k^{c_h}}
	\prod_{\eta = 1}^{\kappa} \frac{u^{d_{h + 1}} x^{\vec{z}_{(\eta)}}}
	{x_k^{c_{h + 1}}}
\end{align}
\noindent
consisting of products of generators $F_h^{\vec{w}_{(\lambda)}}$
and $F_{h + 1}^{\vec{z}_{(\eta)}}$ which are in the form
prescribed in the theorem statement. Since we wrote an arbitrary monomial
$\frac{u^{\beta} x^{\vec{v}}}{x_k^\gamma} \in R_D$ as a product of generators,
this shows that the $F_i^{\vec v}$ generate $R_D$.

Next, we show that the relations given in the statement of the theorem generate all relations. 
In particular, if $j \geq i + 2$ then $F_i^{\vec{v}} F_j^{\vec{w}}$
has a decomposition of the form ~\eqref{eqn:one-point-canonical-form} of 
products of generators in adjacent degrees where $h$
depends on $i$ and $ j$, so we denote $h_{i, j} := h \in \{1, \ldots, r\}$. We 
also have that $i \leq h_{i, j} < j$ since $(d_i + d_j, c_i + c_j)$ is in the
$\bz_{\geq 0}$-span of $(d_{h_{i, j}}, c_{h_{i, j}})$ and
$(d_{h_{i, j} + 1}, c_{h_{i, j} + 1})$. Furthermore,
$h_{i, j} \neq i$ and $h_{i, j} \neq j - 1$
as follows from an analogous proof to that given by O'Dorney
\cite[Theorem 6]{dorney:canonical} for the case of
$\bp^1$. This gives the relations
$G_{i, j}^{\vec{v}, \vec{w}}$.

One can use the relations $G_{i, j}^{\vec{v}, \vec{w}}$
to transform any monomial in the $F_i^{\vec{v}}$'s involving
indices that differ by more than $1$ to a monomial in the $F_i
^{\vec{v}}$'s involving indices that differ by at most $1$.

We also have relations involving generators in consecutive
indices. Suppose $F_i^{\vec{v}}$ and $F_j^{\vec{w}}$ are
generators with $j = i$ or $j = i + 1$ and $\vec{v} \in
\mss_i, \vec{w} \in \mss_j$ with $\vec{v} \not\prec \vec{w}$.
Let $\vec{y}$ and $\vec{z}$ be the unique vectors in $\mss_i$ and
$\mss_j$, respectively, such that $\vec{y} + \vec{z} = \vec{v} +
\vec{w}$ and $\vec{y} \prec \vec{z}$ (i.e.~ the nonzero indices
of $\vec{y}$ followed by those of $\vec{z}$ give an increasing
sequence). Then we see that
\begin{align*}
	F_i^{\vec{v}} F_j^{\vec{w}} = x^{\vec{v} + \vec{w}}
	\left(\frac{u^{d_i}}{x_k^{c_i}}\right)
	\left(\frac{u^{d_j}}{x_k^{c_j}}\right)
	= x^{\vec{y} + \vec{z}}
	\left(\frac{u^{d_i}}{x_k^{c_i}}\right)
	\left(\frac{u^{d_j}}{x_k^{c_j}}\right)
	= F_i^{\vec{y}} F_j^{\vec{z}},
\end{align*}

\noindent 
which give the relations $L_{i, j}^{\vec{v}, \vec{w}}$.

Now, we may apply the relations $L_{i, j}^{\vec{v}, \vec{w}}$
to any monomial in the $F_i^{\vec{v}}$'s involving
indices that differ by at most $1$ to produce the canonical form
\begin{align*}
	\prod_{\tau=1}^\lambda (F_i^{\vec{v}_{(\tau)}})^{g_{\vec{v}_{(\tau)}}}
	\prod_{\eta=1}^\kappa (F_{i+1}^{\vec{v}_{(\eta)}})^{g_{\vec{v}_{(\eta)}}}
\end{align*}

\noindent
where $\vec{v}_{(1)} \prec \vec{v}_{(2)} \prec \ldots \prec
\vec{v}_{(\lambda)} \prec \vec{w}_{(1)} \prec \ldots \prec
\vec{w}_{(\kappa)}$. Consequently, the relations of form
$G_{i, j}^{\vec{v}, \vec{w}}$ and the relations of form
$L_{i, j}^{\vec{v}, \vec{w}}$ generate all the relations among the
$F_{i}^{\vec{v}, \vec{w}}$.
\end{proof}

\begin{rem}
\label{rem:proj-grobner}
	In the case that $D = \alpha_0 D_0$ is supported on a single
	hypersurface and $D' = 0$, the relations in Theorem ~\ref{thm:proj-one-point}
	form a reduced Gr\"{o}bner basis with respect to the ordering
	given in Definition~\ref{defn:vec-order}.
\end{rem}

Theorem ~\ref{thm:proj-one-point} gives an inductive procedure to
compute presentations of effective divisors which are supported on
hyperplanes. However, for $m \geq 2$, there are hypersurfaces which
are not unions of hyperplanes. We now address this general case, giving an inductive presentation of section rings of effective divisors, and a tight bound on the degrees of their generators and
relations.

\begin{proof}[Proof of Theorem~\ref{thm:proj-effective-intro}]
We proceed by induction on $n$.
If $n = 0$, i.e.~ $D = 0$, then we are done.
Now, we inductively add hypersurfaces.
Let $D' = \sum_{i = 0}^{n} \alpha_i D_i \in \di \bp^m \otimes_\bz \bq$, 
and assume the theorem holds for $D'$. It suffices to show
the theorem holds for $D \in \di \bp^m \otimes_\bz \bq$, where $D = D' + \alpha C$ for some
degree $\pdeg$ hypersurface $C$.
If $C$ were a hyperplane, we would then be done, by Theorem ~\ref{thm:proj-one-point}.

To complete the theorem, we reduce the case that $C$ is a general hypersurface
to the case that $C$ is a hyperplane, by using the Veronese embedding.

If $C$ is of degree $\pdeg$, consider the Veronese embedding
$\nu_\pdeg^m \colon \bp^m \rightarrow \bp^{\binom{m+\pdeg}{\pdeg}-1}$
so that the image of $C$ is the intersection of 
a hyperplane in $\bp^{\binom{{m + \pdeg}}{\pdeg} - 1}$ with $\nu_\pdeg^m(\bp^m)$.
Now, the ring $R_{\alpha C}$ is isomorphic to the
$\pdeg$ Veronese subring of 
\[R_{\alpha V(x_0)} = \bigoplus_{d \geq 0} u^d H^0(\bp^m, d \alpha V(x_0)).\]
Therefore, we can bound the degree of generators and relations of
$R_D$ over $R_{D'}$ by the degree of generators and
relations for $R_C \cong R_H.$
This reduces the case of a hypersurface $C$ to
a hyperplane $H$, completing the proof by Theorem
~\ref{thm:proj-one-point}.
\end{proof}

\begin{rem}
\label{rem:exact-effective-bounds}
The proof of Theorem ~\ref{thm:proj-effective-intro} not only gives bounds on the degrees of the generators and relations, but actually gives an explicit method for computing the presentation. 
Also, a minimal generating set of $R_D$ over $R_{D'}$ in 
Theorem~\ref{thm:proj-one-point}
can be given as a subset of the generating set given in 
Theorem~\ref{thm:proj-one-point}. Further,
one can verify a minimal generating
set necessarily contains generators in each of the degrees $d_0, \ldots, d_r$
using the definition of 
Hirzebruch-Jung continued fractions.
In particular,
the bounds of
Theorem ~\ref{thm:proj-effective-intro}
are tight. The generator bound is always achieved and the bound on the relations is achieved if $m \geq 2$.
\end{rem}

Before giving bounds for arbitrary divisors on projective space in Subsection~\ref{ssec:arbitrary-projective},
we give a detailed example of the generators and relations for a
section ring associated to an effective divisor on projective space.
\begin{example}
	\label{example:effective-example}
	In this example, we work out generators and relations 
	for the section ring $R(\bp^2, D)$ with
	\begin{align*}
		D := \frac{1}{6}V(x^2 + y^2 + z^2) + \frac{2}{5}V(x).
	\end{align*}
	We loosely follow the
	algorithm
	described in the course of the 
	proof of Theorem~\ref{thm:proj-effective-intro} and use notation from the statement of Theorem~\ref{thm:proj-effective-intro}.

	For ease of notation, let $h := x^2 + y^2 + z^2$
	and let $D' := \frac{1}{6}V(h)$.
	We first compute generators and relations for $R_{D'}$,
	then compute generators and relations for $R_D$ over $R_{D'}$,
	and finally put them together to obtain generators and relations
	for $R_D$.
	
	To start, we compute generators and relations for $R_{D'}$.
	Indeed, applying \ref{thm:proj-one-point} and scaling
	the degrees by 2, we see that $\frac{1}{6}$ has convergents
	given by
	\begin{align*}
		0 = \frac{0}{1} < \frac{1}{6},
	\end{align*}
	and so the generators are given by 
	\begin{align}
	\begin{aligned}
	\label{equation:d'-generators}
		F_0^{'(0,0,0)} &= u^1 \\
		F_1^{'(2,0,0)} &= \frac{u^6 x^2}{h}, \\
		F_1^{'(1,1,0)} &= \frac{u^6 xy}{h}, \\ 
		F_1^{'(1,0,1)} &= \frac{u^6 xz}{h}, \\
		F_1^{'(0,2,0)} &= \frac{u^6 y^2}{h}, \\
		F_1^{'(0,1,1)} &= \frac{u^6 yz}{h},  \\
		F_1^{'(0,0,2)} &= \frac{u^6 z^2}{h}.
	\end{aligned}
	\end{align}
Further, the relations in $R_{D'}$ are given by 	
\begin{align}	
			\begin{aligned}
	\label{equation:d'-relations}
		L_{1,1}^{'(1,1,0), (1,0,1)} &= \frac{u^6 xy}{h}\frac{u^6 xz}{h} - \frac{u^6x^2}{h}\frac{u^6yz}{h},\\
		L_{1,1}^{'(1,0,1), (0,1,1)} &= \frac{u^6 xz}{h}\frac{u^6 yz}{h} - \frac{u^6xy}{h}\frac{u^6z^2}{h},\\
		L_{1,1}^{'(1,0,1), (0,2,0)} &= \frac{u^6 xz}{h}\frac{u^6 y^2}{h} - \frac{u^6xy}{h}\frac{u^6yz}{h}, \\
		L_{1,1}^{'(1,1,0), (1,1,0)} &= \frac{u^6 xy}{h}\frac{u^6 xy}{h} - \frac{u^6x^2}{h}\frac{u^6y^2}{h}, \\
		L_{1,1}^{'(1,0,1), (1,0,1)} &= \frac{u^6 xz}{h}\frac{u^6 xz}{h} - \frac{u^6x^2}{h}\frac{u^6z^2}{h}, \\
		L_{1,1}^{'(0,1,1), (0,1,1)} &= \frac{u^6 yz}{h}\frac{u^6 yz}{h} - \frac{u^6y^2}{h}\frac{u^6z^2}{h}.
\end{aligned}
\end{align}

	Next, we describe generators and relations for $R_D$ over $R_{D'}$.
	Indeed, in this case, $\frac{2}{5}$ has convergents given by
	\begin{align*}
		0 = \frac{0}{1} < \frac{1}{3} < \frac{2}{5}.
	\end{align*}
	Therefore, the generators of $R_D$ over $R_{D'}$ are given by
	\begin{align}
		\begin{aligned}
		\label{equation:d-generators}
		F_0^{(0,0,0)} &= u^1, \\
		F_1^{(0,1,0)} &= \frac{u^3y}{x}, &F_1^{(0,0,1)} &= \frac{u^3z}{x}, \\
		F_2^{(0,2,0)} &= \frac{u^5y^2}{x^2}, &F_2^{(0,1,1)} &= \frac{u^5yz}{x^2}, &F_2^{(0,0,2)} &= \frac{u^5z^2}{x^2}.
	\end{aligned}
	\end{align}
Note that the generator $F_3^{(1,0,0)} = \frac{u^3x}{x}$ could be included, but it is redundant as it is equal to $(u^1)^3$.
	Similarly, the generators $\frac{u^5x^2}{x^2}, \frac{u^5xy}{x^2}, \frac{u^5xz}{x^2}$ are redundant.
	Furthermore, we have relations given by
	\begin{align}
		\begin{aligned}
		\label{equation:d-relations}
		G_{0,2}^{(0,0,0),(0,2,0)} &= u \cdot \frac{u^5 y^2}{x^2} - \left( \frac{u^3y}{x} \right)^2 \\
		G_{0,2}^{(0,0,0),(0,1,1)} &= u \cdot \frac{u^5 yz}{x^2} -  \frac{u^3y}{x} \frac{u^3z}{x} \\
		G_{0,2}^{(0,0,0),(0,0,2)} &= u \cdot \frac{u^5 z^2}{x^2} - \left( \frac{u^3z}{x} \right)^2 \\
		L_{1,2}^{(0,0,1), (0,2,0)} &= \frac{u^3z}{x}\cdot \frac{u^5y^2}{x^2} - \frac{u^3y}{x}\frac{u^5yz}{x^2} \\
		L_{1,2}^{(0,0,1), (0,1,1)} &= \frac{u^3z}{x}\cdot \frac{u^5yz}{x^2} - \frac{u^3y}{x}\frac{u^5z^2}{x^2} \\
		L_{2,2}^{(0,1,1),(0,1,1)} &= \frac{u^5yz}{x^2}\cdot \frac{u^5 yz}{x^2} - \frac{u^5y^2}{x^2}\frac{u^5z^2}{x^2}.
\end{aligned}
\end{align}

	Then, combining the above, generators for the ring $R_D$ are given by	
	\eqref{equation:d'-generators} together with
	\eqref{equation:d-generators}
	and relations are given by
	\eqref{equation:d'-relations} together with \eqref{equation:d-relations}.
\end{example}

\ssec{Bounds for Arbitrary Divisors on Projective Space}
\label{ssec:arbitrary-projective}
We now offer bounds on generators and relations of $R_D$ for a general $\bq$-divisor $D \in \di \bp^m \otimes_\bz \bq$.
Write \begin{align*}
	D = \sum_{i=0}^{n}\alpha_i D_i \in \di \bp^m \otimes_\bz \bq
\end{align*}
For the remainder of this section, we shall make the additional assumption that 
\begin{align} \label{eqn:pm-basis-assumption} f_0, \ldots, f_{m}
	\text{ are independent linear forms.} \end{align}
This may necessitate the inclusion of ``ghost divisors'' $D_i$ with coefficients $\alpha_i = 0$. 

The main aim of this section is to prove Theorem
~\ref{thm:proj-generators-relations}.
Having justified the necessity of adding ghost divisors, we proceed to bound the
number of generators and relations of arbitrary $\bq$-divisors in 
projective space. 
In Proposition~\ref{prop:pm-span-and-basis}, we record a general proposition describing
a basis for $H^0(\bp^m, dD)$.
We bound the generators in Lemma ~\ref{lem:proj-generators}, and we use Lemmas ~\ref{lem:composite-map}, ~\ref{lem:bound-ker-chi}, and ~\ref{lem:proj-relations-psi} to bound the 
degree of relations in the proof of Theorem ~\ref{thm:proj-generators-relations}. Proposition ~\ref{prop:cone-generation}, Lemma ~\ref{lem:composite-map}, and Lemma ~\ref{lem:bound-ker-chi} are quite 
general and will also be used in Section ~\ref{sec:hirz} to bound the 
degree of generators and relations on Hirzebruch surfaces.
However, before moving on to the proof of ~\ref{thm:proj-generators-relations}, we justify the importance of assumption
~\eqref{eqn:pm-basis-assumption} with several illustrative examples.

\sssec*{The importance of ghost divisors}
\label{sssec:ghost-divisors}

\begin{example}
\label{eg:hyperplane}
In this example, we show that the naive generalization of
~\cite[Theorem 8]{dorney:canonical} of generation in degree at most
$\sum_{i=0}^{n}\ell_i$ cannot possibly hold. The reason for this is
that the divisors may be
expressible as functions in $m$ of the $m+1$ variables on $\bp^m$.

Concretely, take $D = \frac{1}{2}H_0 - \frac{1}{3}H_1$ where $H_0 = V(x_0),
H_1 = V(x_1)$ are two coordinate hyperplanes in $\bp^2$. Then, $R_D$
has generators in degree $2$ and $3$ which can be written as $u^2
\frac{x_1}{x_0}, u^3 \frac{x_1}{x_0}.$ In fact, for all degrees
less than $5$, the elements of $R_D$ can all be expressed as
rational functions in $x_0, x_1$. However, in degree $6$, there is
$u^6 \cdot \frac{x_1^2 x_2}{x_0^3}$. Since this
involves $x_2$, it must be a generator.

This example generalizes slightly to any divisor of the form $D =
\frac{1}{k}H_0 - \frac{1}{k+1}H_1 \in \di \bp^2 \otimes_\bz \bq,$ with
$k \in \bn$, showing that there will always exist a generator in
degree $k(k + 1)$.

This example further generalizes to the following situation:
Suppose 
\[
	D = \sum_{i=0}^{n} \frac{\num_i}{\den_i}D_i \in \di \bp^m \otimes_\bz \bq,
\]
where $\deg D_i = \bida_i$ and $\deg D = \frac{1}{\lcm
_{0 \leq i \leq n}(\den_i \cdot \bida_i)}$. Then, if $D_i = V(f_i)$
where all $f_i$ can be written as a polynomial function in $x_0,
\ldots, x_{m-1},$ it follows that $R_D$ always has a generator in degree
$\lcm(\den_i \cdot \bida_i)$.
\end{example}

As illustrated in Example ~\ref{eg:hyperplane}, 
when all components in the support of divisor can be 
written in terms of $m$ of the $m+1$ variables on $\bp^m$,
we cannot hope to bound the degree of generation
by anything less than the sum of the least common multiples of
the denominators. This issue can easily be circumvented by adding in ``ghost
divisors.'' That is, we may add divisors of the form $0 \cdot H_i$ to
$D$, and reorder so that if $D = \sum_{i=0}^{n}\alpha_i V(f_i)$,
then $f_0, \ldots, f_m$ are independent linear functions in
$x_0,\ldots, x_m$.

\begin{rem}
\label{rem:proj-two-points}
We cannot extend Theorem~\ref{thm:proj-one-point} to the case when
$D$ is supported at two hypersurfaces with arbitrary rational (non-effective)
coefficients in the same manner that O'Dorney does for the
$\bp^1$ case \cite[Section 4]{dorney:canonical}. As shown
in Example ~\ref{eg:hyperplane}, the degrees of generation of the
section ring of a general $\bq$-divisor supported on two
hyperplanes cannot be bounded so tightly. The two-point
$\bp^1$ result leverages the fact that $\bp^1$ has precisely two
independent coordinates, so that two distinct integral subschemes
cannot represent equations of only $m$ of the $m+1$ coordinates.
\end{rem}

In Example ~\ref{eg:radical}, we show that it is still, in general,
necessary to add ghost divisors, even when the irreducible components
of a divisor are not all
expressible as functions in $m$ of the $m+1$ variables on $\bp^m$.

\begin{example}
\label{eg:radical}
Consider $D :=\frac{-1}{5}V(x_0^2 + x_1^2 + x_2^2) + \frac{1}{7}V(x_0^2 + x_1^2 +
x_3^2) + \frac{1}{17}V(x_0^2 + x_2^2 + x_3^2) - \frac{1}{596}V(x_1^2 + x_2^2 +
x_3^2)$. In degree $355216 = 5 \cdot 7 \cdot 17 \cdot 596$, $R_D$ has dimension
$6$.
However, for all $d \leq 355216,$ $h^0(\bp^3, dD) \leq 1$
and 
$h^0(\bp^3, dD)  = 1$ precisely when $\lfloor d D \rfloor = 0$ 
(so in this case, $H^0(\bp^3, d D)$ corresponds to the constant functions).
Therefore $R_D$ has a generator in degree $355216$.
Hence, we cannot hope to bound the degree
of generation of $R_D$ as a linear combination of $\ell_i$,
in analogy to \cite[Theorem 8]{dorney:canonical}
unless we require that $D$ includes ghost divisors. That is, unless
$D_0,\ldots D_m$ are taken to be linearly independent hyperplanes.
\end{example}

\sssec*{A basis for sections on projective space}
\label{sssec:basis-projective}

In order to prepare ourselves to prove Theorem~\ref{thm:proj-generators-relations},
we will need the following simple description of
a basis for $H^0(\bp^m, dD)$.

\begin{prop}
\label{prop:pm-span-and-basis}
Assuming Equation~\ref{eqn:pm-basis-assumption},
the functions $u^d \cdot \prod_{i=0}^n f_i^{c_i} \in (R_D)_d$ \rm{(}recall $u$ is a dummy variable keeping
track of the degree\rm{)}
satisfying
both of the following conditions

\begin{enumerate}
	\item $\sum_{i=0}^{n} c_i \cdot \bida_i = 0$
	\item $c_i \geq - \lfloor d \alpha_i \rfloor$
\end{enumerate}

\noindent
for $c_0, \ldots, c_n \in \bz$
span $H^0(\bp^m, dD)$ over $\bk$. Furthermore, such functions 
that also satisfy
\begin{enumerate}[resume]
	\item $c_i = -\lfloor d\alpha_i \rfloor \text{ for all } i > m$
\end{enumerate}

\noindent
form a basis for $H^0(\bp^m, dD)$ over $\bk$.
\end{prop}

\begin{proof}
By definition of $H^0(\bp^m,dD)$, functions satisfying conditions 
(1) and (2) lie in $H^0(\mathbb{P}^m,dD)$.
Conditions (1) and (2) are also necessary for some monomial in the $f_i$ to lie in
$H^0(\mathbb{P}^m,dD)$.
Since the monomials in the $f_i$ span $R_D$, it follows the monomials satisfying
$(1)$ and $(2)$ span $(R_D)_d = H^0(\bp^m, dD)$.
To complete the proof, it
suffices to check functions satisfying conditions (1)-(3) form
a basis of $H^0(\bp^m,dD)$. There are $\binom{m+ \deg \lfloor dD \rfloor }{m}$
functions satisfying conditions (1)-(3). However we know
$h^0(\bp^m,dD) = \binom{m+ \deg \lfloor dD \rfloor }{m},$ so it suffices to show
that those monomials satisfying conditions (1)-(3) are
independent. 
To see why these are independent, observe that monomials satisfying (1)-(3) are all of the form
$f_0^{c_0} \cdots f_m^{c_m} \cdot g$ for the fixed monomial
$g = f_{m+1}^{-\lfloor d\alpha_{m+1} \rfloor} \cdots f_n^{-\lfloor d\alpha_{m+1} \rfloor}$.
For any fixed $N \in \bz$, the set 
\begin{align*}
 \left\{ f_0^{c_0} \cdots f_m^{c_m} : (c_0, \ldots, c_m) \in \bz^{m+1}, \sum_{i=0}^m c_i = N \right\}	
\end{align*}
forms an independent set over $\bk$.
Therefore, multiplying the monomials in the above set with the $g$ also forms an independent set, and these are precisely the monomials satisfying (1)-(3) 
with $N = -\deg g$.
\end{proof}

\sssec*{Bounding the generators}
\label{sssec:bounding-generators}

We now develop the tools to bound the degrees of generators for $R_D$
for $D$ an arbitrary divisor on $\bp^m$.

Recall that for a semigroup $\Sigma \subset \bz^m$ we say $e_0, \ldots, e_n \in \Sigma$ are a set of {\bf extremal rays}
for $\Sigma$ if 
$\Sigma$ is contained in the $\bq_{\geq 0}$ span of
$e_0, \ldots, e_n$.

\begin{prop}
\label{prop:cone-generation}
Let $n \in \bz,$ let $\alpha_0, \ldots, \alpha_n \in \bq$, and let
$a_i, b_i \in \bz$ with $0 \leq i \leq n.$ Define
\begin{align*}
	\Sigma := \left \{(d, c_0, \ldots, c_n) \in \bz^{n + 2} \colon c_i \geq -
	d \alpha_i, 0 \leq i \leq n \text{ and } \sum_{i = 0}^{n} a_i =
	\sum_{i	= 0}^{n}b_i = 0 \right \}.
\end{align*}

\noindent
Suppose $e_0, \ldots, e_t \in \Sigma$ with $e_i = (\pdeg_i, c_0^i,
\ldots, c_n^i)$ are a set of extremal rays of $\Sigma$.

Then, as a semigroup, $\Sigma$ is generated by
elements whose first coordinate is less than $\sum_{i = 0}^{t}
\pdeg_i$. Furthermore, every element $\sigma \in \Sigma$ can be
written in a canonical form 

\begin{equation}
	\label{eqn:sigma-canonical-form}
	\sigma = \lambda + \sum_{i = 0}^{t} \zeta_i e_i
\end{equation}

\noindent
with $\zeta_1, \ldots, \zeta_t \in \bz_{\geq 0}$, $0 \leq s_i < 1$, and $\lambda = \sum_{i = 0}^{r} s_i e_i$ 
so that the
first coordinate of $\lambda$ is less than $\sum_{i=0}^{t}\pdeg_i$.
\end{prop}

\begin{proof}
By assumption, $\sigma \in \Sigma$ can be written as $\sigma = \sum_
{i = 0}^{t} r_i e_i$ with $r_i \in \bq$. Let $\fr(r) := r - \lfloor r
\rfloor$ denote the fractional part of $r$. Let $\lambda = \sum_{i = 0}
^{t} \fr(r_i)$. Whence, we can write $\sigma = \lambda + \sum_{i = 0}
^{t} \lfloor r_i \rfloor e_i.$ Consequently, $\sigma$ lies in the
$\bz_{\geq 0}$ span of $\lambda, e_0, \ldots, e_t$, which all have
first coordinate less than $\sum_{i=0}^{t} \pdeg_i$. Ergo, $\Sigma$ is
generated by elements whose first coordinate is less than
$\sum_{i = 0}^{t} \pdeg_i$.
\end{proof}

By Proposition ~\ref{prop:cone-generation}, in order to bound
the degree of generation of $R_D$, we only need bound
the degrees of extremal rays of an associated cone. We now carry out
this strategy.

\begin{lem} \label{lem:proj-generators}
Let $D = \sum_{i=0}^{n} \alpha_i D_i \in \di \bp^m \otimes_\bz \bq,$ where
$\deg D_i = \bida_i$, $\alpha_i = \frac{c_i}{k_i}\in \bq$, and
$\ell_i = \lcm_{j \neq i} (k_j)$. Then, $R_D$ is generated in degrees at most $\sum_{i=0}^n \ell_i \bida_i.$
\end{lem}
\begin{proof}
Let 
\begin{align}\label{eqn:Sigma-defn}
	\Sigma = \left \{(d, c_0, \ldots, c_n) \in \bz^{n+2} \colon c_i \geq - d
\alpha_i, \; 0 \leq i \leq n, \text{ and } \sum_{i=0}^{n} \ell_i \bida_i = 0
\right \}.
\end{align}

Observe that $\Sigma$ has extremal rays given by the lattice points 
\begin{equation}\label{eqn:e-i-proj}
	e_i = \left(\ell_i \bida_i, - \alpha_0 \ell_i \bida_i, \ldots
-\alpha_{i-1} \ell_i \bida_i, \ell_i \sum_{j\ne i} \alpha_j \bida_j,
-\alpha_{i+1} \ell_i \bida_i, \ldots, -\alpha_n, \ell_i \bida_i \right)
\end{equation}
for each $i\in \{0, \ldots n\}$.
Therefore, applying Proposition ~\ref{prop:cone-generation}, we see $R_D$ is generated in degrees less than
\[
	\sum_{i=0}^n \ell_i \bida_i.
\qedhere
\]\end{proof}

Let $w_1, \ldots w_r$ be the generators in degrees at most $\sum_{i=0}^n \ell_i
\bida_i$ (given by Lemma ~\ref{lem:proj-generators}), and let 
$\phi \colon \bk[w_1, \ldots w_r] \to R_D$ be the natural surjection.
For the remainder of the section, we aim to bound the degree of relations
of $R_D$, or equivalently, the degree of generation of $\ker \phi$.
We can factor $\phi$ through the
semigroup ring 

\[
	\bk[\Sigma] =  \langle u^d z_0^{c_0} \cdots z_n^{c_n} \colon c_i \in
\bz, \; c_i \geq -d \alpha_i, \mbox{ and }\sum_{i=0}^{n} \bida_i c_i =
\sum_{i=0}^{n} \bidb_i c_i \rangle. 
\]
by
\begin{equation}
\label{eqn:factor-through-semigroup-ring}
\begin{tikzcd}[baseline=(current  bounding  box.center), row sep = tiny]
\bk[w_1,\ldots, w_r] \ar {r}{\chi} & \bk[\Sigma] \ar {r}{\psi} & R_D \\
w_i \ar[mapsto]{r} & u^{d_i}z_0^{c_{i0}} \cdots z_n^{c_{in}} \ar[mapsto]{r} & u^{d_i}f_0^{c_{i0}} \cdots f_n^{c_{in}}.
\end{tikzcd}
\end{equation}

\sssec*{Bounding the relations}
\label{sssec:relations-projective}

We now move on to bounding the degree of relations of a section ring $R_D$ for $D$ a divisor on $\bp^m$.
In Lemma ~\ref{lem:composite-map} we show that the degree of the
generators for $\ker \phi$, which is the same as the degree of relations
of $R_D$, is bounded by the maximum of the degree of generators for
$\ker \chi$ and for $\ker \psi$.
In Lemma ~\ref{lem:bound-ker-chi}, we bound the degree of generation of
$\ker \chi $, and we bound the degree of generation of $\ker \psi$ in Lemma ~\ref{lem:proj-relations-psi}

\begin{lem}
\label{lem:composite-map}
Let $X$ be a $\bk$ variety and let $D = \sum_{i=0}^{n}\alpha_i D_i\in \di X \otimes_\bz \bq$ where $D_i = V(f_i)$. Suppose we have a
surjection $\phi\colon \bk[w_1,\ldots, w_r] \rightarrow R_D$ given by $w_i
\mapsto p_i(f_0, \ldots, f_n),$ where $p_i$ is a monomial in $f_0,\ldots, f_n$.
Let $\bida_0, \ldots, \bida_n, \bidb_0, \ldots, \bidb_n \in \bz_{\geq 0}.$
Then, define
\begin{align*}
	\Sigma = \langle u^d z_0^{c_0} \cdots z_n^{c_n} : c_i \geq -d \alpha_i, \sum_{i=0}^{n} \bida_i c_i = \sum_{i=0}^{n} \bidb_i c_i = 0 \rangle. 
\end{align*}
In this case, we can factor $\phi$ as a composition of $\chi$ and $\psi$ defined by
\[
\begin{tikzcd}[row sep = tiny]
\bk[w_1,\ldots, w_r] \ar {r}{\chi} & \bk[\Sigma] \ar {r}{\psi} & R_D \\
w_i \arrow[mapsto]{r} & u^{d_i}z_0^{c_{i0}} \cdots z_n^{c_{in}} \arrow[mapsto]{r} & u^{d_i}f_0^{c_{i0}} \cdots f_n^{c_{in}}.
\end{tikzcd}
\]
Assuming $\chi$ is surjective, the minimal degree of generation of $\ker \phi$
is at most the maximum of the minimal degree of generation of $\ker \chi$ and
the minimal degree of generation of $\ker \psi$.
\end{lem}

\begin{proof}
First, surjectivity of $\chi$ implies we have an exact sequence
\[
\begin{tikzcd}
0 \ar {r} & \ker \chi \ar{r} & \ker \phi \ar {r} & \ker \psi \ar r & 0.
\end{tikzcd}
\]
This shows that lifts of generators of $\ker \psi$ together with images of
generators of $\ker \chi$ generate all of $\ker \phi,$ as desired. 
\end{proof}

\begin{lem}
\label{lem:bound-ker-chi}
Retaining the notation of Lemma ~\ref{lem:composite-map}, if $\Sigma$ has
extremal rays $e_0,\ldots, e_t$ in degrees $d_0, \ldots, d_t$ then $\ker \chi$
is generated in degrees at most $2(\sum_{i=0}^{t}d_i-1)$.
\end{lem}
\begin{proof}
Since $e_0, \ldots, e_t$ are extremal rays, Proposition
~\ref{prop:cone-generation} implies every element $\sigma \in \Sigma$ can be written
in a canonical form $\lambda + \sum_{i=0}^{t} \zeta_i e_i$ where all $\zeta_i \in \mathbb{Z}_{\ge 0}$.  Let
$\lambda_0 :=
0,\lambda_1, \ldots, \lambda_r$ be all elements of $\Sigma$ which
can be
written in the form $\lambda_j = \sum_{i=0}^{t}s_i e_i$ with $0
\leq s_i < 1.$ Then, for any $1 \leq j \leq k \leq r,$ we can write $\lambda_j + \lambda_k$ in
the above canonical form, yielding a (possibly trivial) relation in degree at most $\deg \lambda_j
+ \deg \lambda_k\leq 2 \cdot \left( \sum_{i=0}^{t}d_i -1 \right).$
Furthermore, these relations generate all relations, as one can apply a
sequence of these relations to put any $\sigma \in \Sigma$ into canonical form
$\sigma = \lambda + \sum_{i=0}^{t}\zeta_i e_i$ from Proposition
~\ref{prop:cone-generation}.
\end{proof}
\begin{lem}
\label{lem:proj-relations-psi}
Let $D = \sum_{i=0}^{n} \alpha_i D_i \in \di \bp^m \otimes_\bz \bq$, where
$\deg D_i = \bida_i$, $\alpha_i = \frac{c_i}{k_i}\in \bq$, and
$\ell_i = \lcm_{j \neq i} (k_j)$. Define $\Sigma$ as in Equation
~\eqref{eqn:Sigma-defn} and $\psi$ as in Equation
~\eqref{eqn:factor-through-semigroup-ring}. Then, $\ker \psi$ is
generated in degrees at most
\begin{align}
\label{eqn:proj-relation-degree}
	\frac{\max_{0 \le i \le n}(\bida_i)}{\deg(D)} +  \sum_{i=0}^n \ell_i a_i.
\end{align}

\end{lem}

\begin{proof}
We claim there exist $\beta_0, \ldots, \beta_n \in \bk[\Sigma]$
such that $\ker \psi$ is generated by
\begin{equation}
\label{eqn:relations-psi-proj}
	u^d(z_i - \beta_i)\prod_{j=0}^n {z_j}^{c_{j}}
\end{equation}

\noindent
for all $d \in \mathbb{N}$ and $c_i \ge -\alpha_i d$ satisfying
$\bida_i + \sum_{j = 0}^n \bida_j c_j = 0$.

Indeed, define the $\beta_i$ as a polynomial in $z_0, \ldots, z_m$ such that $\psi(\beta_i) = \psi(z_i) = f_i \in R_D.$ This is possible by Proposition ~\ref{prop:pm-span-and-basis}.
Furthermore, the relations given in Equation
~\eqref{eqn:relations-psi-proj} generate all relations, since they
allow us to reduce any $u^d \prod_{j = 0}^n z_j^{c_i}$ to a canonical
form, with $c_i = -\lfloor  d \alpha_i\rfloor$ whenever $i  > m$.

For the remainder of the proof, fix $i \in \{0,\ldots, n\}$. To
complete the proof, it suffices to bound the degree of generation
of the relations of the form $u^d(z_i - \beta_i) \prod_{j = 0}^n
z_j^{c_j}$, by Equation ~\eqref{eqn:proj-relation-degree}. For
a given monomial
$u^d(z_i - \beta_i)\prod_{j=0}^n z_j^{c_j} \in \bk[\Sigma]$, we associate it with the corresponding element $(d, c_0,\ldots, c_n) \in\Sigma$. Let $\Sigma_i \subseteq \bz^{n + 2}$
be the set of points of the form $(d, c_0, \ldots, c_n)$ satisfying
$c_j \ge -d \alpha_j$ for all $j$ and $\sum_{j=0}^n c_j a_j = -a_i$.
Let
\begin{align*}
	\delta_i := \left(\frac{a_i}{\deg(D)}, -\frac{\alpha_0 a_0}{\deg(D)},
	\ldots, - \frac{\alpha_n a_n}{\deg(D)} \right).
\end{align*}

\noindent
Then we see $\Sigma_i = \{\sigma \in \Sigma | \sigma - \delta_i \in \newspan_{\bq_{\geq 0}}
(e_0, \ldots, e_n)\}$ with $e_i$ as defined in Equation
~\ref{eqn:e-i-proj}. Therefore, we can write any element of
$\Sigma_i$ uniquely as
\[
	\delta_i + \sum_{j=0}^n c_j e_j
\]
where $c_j \in \mathbb{R}$ for each $j$.

Whenever some there is some $j$ for which $c_j \ge 1$, we can write the relation $u^d(z_i -
\beta_i)\prod_{j=0}^n z_j^{c_j} = e_j h,$ for some
$h \in \Sigma_i$. Therefore, for a fixed $i$, relations of the form
$u^d(z_i - \beta_i)\prod_{j=0}^n z_j^{c_j} \in \bk[\Sigma]$ are
generated by those in degrees less than
\[
	\frac{\bida_i}{\deg(D)} + \sum_{i=0}^n \ell_i a_i,
\]
as $\deg \delta_i = \frac{a_i}{\deg(D)}$. Hence,
$\ker \psi$ is generated in degrees less than
\[
	\frac{\max_{0 \leq i \leq n} \bida_i}{\deg(D)} + \sum_{i=0}^n \ell_i a_i.
\]
\end{proof}

\sssec*{Proving Theorem ~\ref{thm:proj-generators-relations}}
\label{sssec:proving-theorem-proj}
By combining the above results, we get our main theorem bounding
the generator and relation degrees of the section ring of any
$\bq$-divisor on projective space.


\begin{proof}[Proof of Theorem~\ref{thm:proj-generators-relations}]
The bound on degree of generation is precisely the content of Lemma 
~\ref{lem:proj-generators}. It only remains to bound the degree of 
relations.

By ~\ref{lem:bound-ker-chi}, $\ker \chi$ is generated in degrees 
at most $2\sum_{i=0}^n \ell_i a_i$ and by Lemma
~\ref{lem:proj-relations-psi}, $\ker \psi,$ is generated in
degrees up to $\frac{\max_{0 \leq i \leq n} \bida_i}{\deg(D)} + \sum_{i=0}^n \ell_i a_i$. 
Consequently, Lemma ~\ref{lem:composite-map} implies that $\ker \phi$
is generated in degrees less than
\[
	\max \left(2 \sum_{i=0}^n \ell_i a_i, \frac{\max_{0 \le i \le n}
	(\bida_i)}{\deg(D)} + \sum_{i=0}^n \ell_i a_i \right).
	\qedhere
\]
\end{proof}

\begin{rem}
\label{rem:exact-noneff-bounds}
The bounds given in Theorem 
~\ref{thm:proj-generators-relations} are asymptotically
tight to within a factor of two 
for the following class of divisors. Consider
a divisor 
$ D = \sum_{i=0}^n\frac{p_i}{2q_i}H_i \in \di \bp^m \otimes_\bz \bq$
such that $H_i$ are hyperplanes,  
$q_i$ are pairwise coprime integers, and $p_i$ are chosen
so that $\deg D = \frac{1}{2 \prod_{i=0}^n q_i}$. Further,
choose a linear subspace $\pi: \bp^1 \rightarrow \bp^m$  generically so that
$\pi^* D = \sum_{i=0}^n \frac{p_i}{2q_i}P_i$, 
where $P_i$ are distinct points in $\bp^1$. To choose such a map
$\pi$, we may need to assume that the base field is infinite.
By
\cite[Remark, p. 9]{dorney:canonical},
the given bounds on the generators and relations of
$R_{\pi^*D}$ are within
a factor of two of the degree of generation and relations of
$R_{\pi^*D}$. Finally, since restriction map
$R_D \rightarrow R_{\pi^*D}$ induced by the restriction maps
on cohomology $H^0(\bp^m, dD) \rightarrow H^0(\bp^1, \pi^*(dD))
\cong H^0(\bp^1, d\pi^* D)$ are surjective, we obtain that
the bounds for the generators and relations of $R_D$ given in
Theorem ~\ref{thm:proj-generators-relations} also
agree with the degree of generation and relations 
to within a factor of two.
\end{rem}

\section{Section Rings of Hirzebruch Surfaces}
\label{sec:hirz}
\invisiblesubsection
Let $\hirz_m$ denote the $m$-th Hirzebruch surface.
The aim of this section is to prove Theorem
~\ref{thm:hirz-generators-relations}, which bounds the degree of
generators and relations of the section ring of any $\bq$-divisor
on $X = \hirz_m$, for all $m \geq 0$.

One way to describe the Hirzebruch surface $\hirz_m$ is as a
quotient,
\[
\hirz_m \cong (\ba^2 \setminus \{0\}) \times (\ba^2 \setminus \{0\})/\mathbb G_m \times \mathbb G_m
\]
where $\mathbb G_m$ is the multiplicative subgroup of $\ba^1$, and the action of $\mathbb G_m \times \mathbb G_m$ is given by
$(\lambda, \mu) \cdot (u\colon v; z\colon w) \mapsto (\lambda u\colon \lambda v; \mu z\colon \lambda^{-m} \mu w),$ as described in \cite[p.~ 6]{zhao:counting-cubic}. Hence, one can think of
$\hirz_m$ as a $\bp^1$ bundle where $u,v$ are the coordinates on
$\bp^1$ and $z,w$ are the coordinates on the fiber.
Sections of a line bundle $\mathscr L$ on $\hirz_m$ can be written as
rational functions in $z, w, u, v$.
Furthermore we define the bi-degree of a monomial $u^a v^{b} z^c w^d$
on $\hirz_n$ to be
$(a + b + mc, c + d)$. Rational sections of $\hirz_m$ can be written as
rational functions with numerators and denominators of the same bi-degrees. To see this, 
observe
$\Pic(\hirz_m) \cong \bz \times \bz$, where the class of a line bundle
in $\Pic(\hirz_m)$
is determined by its bi-degree, as follows from the excision exact sequence
for class groups.

Furthermore, we will restrict to the case that $D$ is a divisor
for which both of its bi-degrees are positive. We now justify this
restriction. If either of the bi-degrees of $D$ are negative,
then the section ring is concentrated in degree 0. If one of the
bi-degrees is 0, say the first one is 0, then $R_D$ is isomorphic
to $R_{D'}$, where $D' \in \di \bp^1 \otimes_\bz \bq$, 
where $D'$ can be written as a sum of divisors whose degrees
are multiples of the second bi-degree of $D$. Since the case of
$\bp^1$ has already been analyzed in \cite{dorney:canonical}, we
are justified in assuming that both bi-degrees of $D$ are positive.

For the remainder of this section we will assume $D_1, D_2, D_3,$ and $D_4$ are
distinct divisors with bi-degrees $(1,0)$
, $(1,0)$, $(0,1)$, and $(0,1)$ respectively with $D_i = V(f_i)$ for $1 \leq i \leq 4$ with $f_i \in \sco(\bida_i, \bidb_i)$.  
In order to achieve the above condition on the bi-degrees of
$D_1, \ldots D_4$, it may be necessary to add in ``ghost divisors'' (i.e.
divisors with of the desired form with a coefficient $0$).
Also, $f_1$ and $f_2$ are independent linear
polynomials in $u$ and $v$ and $f_3$ and $f_4$ are independent
linear polynomials in $z$ and $w$.
Analogously to Proposition ~\ref{prop:pm-span-and-basis} for the case
of $\bp^m$, all rational functions on $\hirz_m$
can be written uniquely in a form where their numerator is a function
of only $f_1,f_2,f_3$, and $f_4$.

\begin{defn}
\label{defn:t-defn}
Define 
\begin{equation*}
	\Te(D) = \left\{i \in \{1, \ldots, n\}: \bida_i \sum_{k=1}^n \bidb_k 
\alpha_k = \bidb_i \sum_{k=1}^n \bida_k \alpha_k \right\},
\end{equation*}

\begin{equation*}
	\Tp(D) = \left\{ i \in \{1, \ldots, n\}\colon  \vphantom{\sum_{k = 1}^n} 
	\bida_i \sum_{k = 1}^n \alpha_k \bidb_k > \bidb_i \sum_{k = 1}^n \alpha_k \bida_k 
\right\},
\end{equation*}

\noindent
and
\begin{equation*}
	\Tm(D) = \left\{ i \in \{1, \ldots, n\}\colon \bida_i \sum_{k = 1}^n \alpha_k
	\bidb_k < \bidb_i \sum_{k=1}^n \alpha_k \bida_k \right\}.
\end{equation*}
\end{defn}

\begin{lem}
\label{lem:hirz-generators}
For $D = \sum_{i=1}^{n} \frac{c_i}{k_i}D_i \in \di \hirz_m \otimes_\bz \bq$, with $\deg D_i = (\bida_i, \bidb_i), \ell_i = \lcm_{j\neq i} (k_j), \ell_{i,j} = \lcm_{h \neq i,j} (k_h).$ Then, the section ring $R_D$ is generated in degrees at most
\begin{equation}\label{eqn:def-sigma}
	\rho := \sum_{i \in \Te(D)} \gcd(\bida_i, \bidb_i) \ell_i + \sum_{\substack
	{i \in \Tp(D) \\ j\in \Tm(D)}} (\bida_i \bidb_j - \bida_j \bidb_i)
	\ell	_{i,j}.
\end{equation}
\end{lem}

\begin{proof}
Suppose $g
\in (R_D)_d$ is a monomial.
 Then 
\[
	g = u^d \prod_{i = 1}^n {f_i}^{c_i}
\]

\noindent
for some $c_i \ge - \alpha_i d$ such that $\sum_{i=1}^n c_i \bida_i = 0$
and $\sum_{i=1}^n c_i \bidb_i = 0$. We can view $g$ as an element $(d, c_1, \ldots, c_n)$ of the
lattice 
\begin{align*}
	\Sigma = \left\{ (d', c_1', \ldots, c_n') \in \bz_{\geq 0}^{n+1} \colon c_i' \ge - d
\alpha_i \text{ for all }i \text{ and }\sum_{i=1}^n c_i' \bida_i = \sum_{i =
1}^n c_i' \bidb_i = 0 \right \}.	
\end{align*}

In order to determine a generating
set for $(R_D)_d,$ it suffices to find the extremal rays of $\Sigma.$
To do this, we extend the method of O'Dorney \cite[Theorem 8]{dorney:canonical}. 
We first consider the sub-cone $\Sigma_1 \subset \Sigma$
given by
\begin{align*}
	\Sigma_1 = \left\{ (d, c_1, \ldots, c_n) \in \bz_{\geq 0}^{n+1} \colon c_i \ge -
d \alpha_i \text{ for all }i \text{ and }\sum_{i=1}^n c_i (\bida_i+\bidb_i) = 0
\right\},
\end{align*}
which
has extremal rays given by
 
\[
	\epsilon_i := \left(1, -\alpha_1, \ldots, -\alpha_{i-1}, \frac{\sum_{j \ne i}
	\alpha_j (\bida_j + \bidb_j)}{\bida_i + \bidb_i}, -\alpha_{i + 1},
	\ldots, -\alpha_n \right).
\]
for $1 \leq i \leq n$.

Let $\Sigma_1 \otimes_\bz \bq$ be the $\bq_{\ge 0}$ span of $\epsilon_1, \ldots \epsilon_n$.
We can intersect $\Sigma_1 \otimes_\bz \bq$ with the hyperplane $H :=
V(\sum_{i=1}^n \bida_i x_i)$  to get the subspace $\Sigma \otimes_\bz \bq = H \cap (\Sigma_1
\otimes_\bz \bq)$.  Then, the extremal rays of $\Sigma$ are precisely the extremal 
rays of $\Sigma \otimes_\bz \bq$.

The extremal rays of $\Sigma \otimes_\bz \bq$ can be represented by points
lying only on the edges 
$\overline{e_i e_j}$.
The extremal rays are given by multiples of those $\epsilon_i$'s
which are contained in $H$ together with intersection points $e_{i, j}$ which
can be expressed as $H \cap \overline{e_i e_j},$ where $i \neq j$ and $e_i,
e_j \notin H$. In this case, $e_{i,j}$ is only defined when
$\# \{H \cap \overline{e_i, e_j}\} = 1$.

From this geometric description of the extremal rays, we can write
the extremal rays algebraically as follows.
For $i \in \Te(D)$, define $e_i \in \bk[\Sigma]$ in degree 
\[
	d_i = \ell_i \gcd(a_i, b_i)
\]
by 
\[
	e_i := d_i \epsilon_i.
\]

\noindent
For $i \in \Tp(D)$ and $j \in \Tm(D)$, with $i < j$, define $e_{i,j}\in \bk[\Sigma]$ in degree 
\[
	d_{i,j} = \ell_{i, j}(\bida_i \bidb_j - \bida_j \bidb_i)
\]
by
\begin{align*}
e_{i,j} := &\frac{a_j\sum_{k\ne i,j} d_{i,j}\alpha_k b_k - b_j\sum_{k \ne i,j} d_{i,j} \alpha_k a_k}{a_ib_j - b_i a_j} \epsilon_i \\
&+ \frac{a_i \sum_{k\ne i,j} d_{i,j} \alpha_k b_k - b_i\sum_{k\ne i,j} d_{i,j} \alpha_ka_k}{a_jb_i - b_j a_i} \epsilon_j.
\end{align*}

Since these $e_i$ are multiples of $\epsilon_i$ and these
$e_{i,j}$ are points of intersection of $H$ with
$\overline {e_ie_j}$ such that neither $e_i$ nor $e_j$
are contained in $H$, these form a set of extremal rays of $\Sigma$.

Thus, Proposition ~\ref{prop:cone-generation} implies that 
$R_D$ is generated in 
degrees less than the sum of the degrees of the $e_i$ and $e_{i,j}$, which is
\[
	\rho = \sum_{i\in \Te(D)} \gcd(\bida_i, \bidb_i)\ell_i + \sum_{\substack{
	i \in \Tp(D) \\	j \in \Tm(D)}} (\bida_i \bidb_j- \bida_j \bidb_i)\ell_{i,j}.
\qedhere
\]\end{proof}

Let $w_1, \ldots, w_r$ be the generators of $R_D$ in degrees less than $\rho$
(as given by Lemma ~\ref{lem:hirz-generators}), and let $\phi$ be the
surjection $\bk[w_1, \ldots w_r] \to R_D$. As in Section ~\ref{sec:proj}, we can factor $\phi$ through the
semigroup ring 
\[
	\bk[\Sigma] = \left \langle u^d z_1^{c_1} \cdots z_n^{c_n} \colon c_i \geq -d
	\alpha_i, \sum_{i=1}^{n} \bida_i c_i = \sum_{i=1}^{n} \bidb_i c_i
	\right \rangle
\]

\noindent
by
\[
\begin{tikzcd}[row sep = tiny]
	\bk[w_1,\ldots, w_r] \ar {r}{\chi} & \bk[\Sigma] \ar {r}{\psi} & R_D \\
	w_i \ar[mapsto]{r} & u^{d_i}z_1^{c_{i1}} \cdots z_n^{c_{in}} \ar[mapsto]{r} & u^{d_i}f_1^{c_{i1}} \cdots f_n^{c_{in}}.
\end{tikzcd}
\]

By Lemma ~\ref{lem:bound-ker-chi}, we can bound the degree of generation of $\ker \chi$ below
$2 \rho$.
Finally, we calculate the degree of generation of $\psi$:

\begin{lem}
\label{lem:hirz-bound-ker-psi}
Let $\rho$ be as in Equation ~\eqref{eqn:def-sigma}. Then, $\ker \psi$ is generated in degrees less than
\[
	\tau := \rho
	+ \max \left(\max_{i\in \Te}(\ell_i \gcd(a_i, b_i)), \max_{\substack{
	i \in	\Tp \\ j\in \Tm}} (\ell_{i,j} (\bida_i \bidb_j - \bida_j \bidb_i))
	\right).
\]
\end{lem}

\begin{proof}
We first claim that there exist $u^{\deg z_1}\beta_1, \ldots, u^{\deg z_n} \beta_n\in \bk[uz_1, uz_2, uz_3, uz_4]$
such that $\ker \psi$ has relations of the form
\[
	u^d(z_i - \beta_i)\prod_{j=1}^n {z_j}^{c_{j}}
\]

\noindent
lying in some degree $d \in \mathbb{N}$
with $c_j \ge -\alpha_j d$ (for all $j$) satisfying $\bida_i + \sum_{j = 1}
^n \bida_j c_j = 0$ and $\bidb_i + \sum_{j=1}^n \bidb_j c_j = 0$.
Specifically, $\beta_i$,
is the polynomial $\beta_i(z_1,z_2,z_3,z_4)$ so that
$\beta_i(z_1,z_2,z_3,z_4) - z_i \in \ker \psi$. Such an element
$\beta$ exists and is unique because rational functions
whose numerators are polynomials in $f_1, f_2, f_3, f_4$ form a basis
of all rational functions in $R_D$. Furthermore, these
generate all relations, since they allow us to reduce any monomial $u^d\prod_{j =
1}^n z_j^{r_j}$ to the canonical form where $r_j = -\lfloor d \alpha_i
\rfloor$ whenever $j > 4$. To bound the degree of generation of these relations,
we bound the degree of generation of the ideal $(\beta_i - z_i) \cap
\ker(\psi)$ for each $i$.

For the remainder of this proof: we fix $i \in \{1, \ldots,
n\}$ and fix a relation of the form $u^d (z_i - \beta_i) \prod_{j =
1}^n z_j^{c_j}$ as we seek to bound the degree of generation of the
ideal $(\beta_i - z_i) \cap \ker(\psi)$. There are no
relations for $i\in \{1, 2, 3, 4\}$ as then $\beta_i - z_i = 0 \in
\bk[\Sigma]$. Thus we restrict attention to $i \ge 5$. Our goal is
to show that if $u^d(z_i - \beta_i)\prod_{j=1}^n z_j^{c_j}$ has
sufficiently high degree, then there is another relation dividing
it. We do so by considering the lattice points corresponding to the
monomials appearing in the relation $u^d(z_i - \beta_i)\prod_{j=1}^
n z_j^{c_j}$, and finding a fixed $\lambda\in \Sigma$ that we can
simultaneously factor out of all monomials.

Consider a relation $u^d(z_i - \beta_i)\prod_{j=1}^n z_j^{c_j}$ and
let it correspond to the lattice point $\sigma := (d, c_1, \ldots,
c_{i-1}, c_{i}+1, c_{i+1}, \ldots, c_n) \in \Sigma$.
Then we can write $\sigma$ as a sum of $s_j e_j$'s for $j\in
\Te$ and $s_{j,k}e_{j,k}$ for $j\in \Tp, k\in \Tm$.  
For convenience, define $d_j := \deg e_j$ (when it exists) and let the $j^{th}$
component of $e_j$ be $-\alpha_j d_j + \kappa_j$ for some $\kappa_j \in \bq$. Also, let $d_{j,k} := \deg e_{j,k}$ (when it exists) and let the $j^{th}$ component of $e_{j,k}$ be $-\alpha_j d_{j,k} + \kappa_{j,k}'$ and the $k^{th}$
component be $-\alpha_k d_{j,k} + \kappa_{j,k}''$ for $\kappa_{j,k}', \kappa_{j,k}'' \in \bq$.

Since $u^d (z_i - \beta_i)\prod_{j=1}^n z_j^{c_j}$ is a relation, each monomial of it must be an element of $\Sigma$. This implies that $s_i\kappa_i\ge 1$, $\sum_{j\in \Tm} s_{i,j}\kappa_{i,j}' \ge 1$, and $\sum_{j\in \Tp} s_{i,j}\kappa_{j,i}'' \ge 1$ if $i\in \Te$, $i\in \Tp$, and $i\in \Tm$ respectively.

If $i\in \Te$ define $r_i := \frac{1}{\kappa_i}$.  If $i\in \Tp$ choose $r_{i,j}\in \bq_{\ge 0}$ for all $j\in \Tm$ such that $\sum_{j\in \Tm} r_{i,j}\kappa_{i,j}' = 1$; similarly, if $i\in \Tm$ choose $r_{j,i}\in \bq_{\ge 0}$ for all $j\in \Tp$ such that $\sum_{j\in \Tm} r_{j,i}\kappa_{j,i}'' = 1$.  For $j\ne i$, define $r_j := 0.$ For all pairs $(j,k)$ so that $j \neq i$ and $k \neq i$ define $r_{j,k} := 0$.
Define $E$ by
\begin{equation}\label{eqn:hirz-E-translation}
	E := \sum_{j\in \Te} (s_j - \lfloor s_j - r_j \rfloor) e_j + \sum_{\substack{j \in 
	\Tp \\ k \in \Tm}} (s_{j,k} - \lfloor s_{j,k} - r_{j,k} \rfloor) e_{j,k}.
\end{equation}

Then,
\begin{align}
\label{eqn:e-bound}
	\deg(E) \le \rho + \begin{cases}
	\ell_i \gcd(\bida_i, \bidb_i)	&\mbox{ if } i \in \Te \\
	\max_{j \in \Tm} \bigl(\ell_{i, j} (\bida_i \bidb_j - \bida_j \bidb_i)\bigr)
	&\mbox{ if } i \in \Tp \\
	\max_{j \in \Tp} \bigl(\ell_{j, i} (\bida_j \bidb_ i - \bida_i \bidb_j) \bigr)
	&\mbox{ if } i \in \Tm. \end{cases}
\end{align}

\noindent
where $\rho$ is as in Equation ~\eqref{eqn:def-sigma}. To obtain the
bound given in Equation ~\eqref{eqn:e-bound}, the $\rho$ term 
corresponds to the sums of fractional parts of $s_j - r_j$'s and $s_
{j, k} - r_{j, k}$'s whereas the second term corresponds to the sums 
of $r_j$'s and $r_{j, k}$'s (noting that in Equation
~\ref{eqn:hirz-E-translation}, $s_j - \lfloor s_j - r_j \rfloor =
r_j + \fr(s_j - r_j)$).

Define
\[
	\lambda := \sigma - E = \sum_{j \in \Te} (\lfloor s_j - r_j \rfloor)
	e_j + \sum_{\substack{j \in \Tp \\ k \in \Tm}} ( \lfloor s_{j, k} -
	r_{j, k} \rfloor) e_{j, k} \in \Sigma.
\]

\noindent
Let $M_i$ be the set of monomials terms of $\beta_i = \beta_i(z_1, \ldots z_4)$.  Let $\mu = \prod_{j=1}^4
{z_j}^{h_j}\in M_i$ and consider the lattice point 
\[
	\sigma_\mu = (d, c_1 + h_1, \ldots, c_4+h_4, c_5, \ldots, c_n).
\]
where $d = \deg \sigma$.
Define 
\[
	E_\mu = \sigma_\mu - \lambda.
\]
From the definitions, $E_\mu$ lies in $\Sigma$ 
and has the same degree as $E$.

By construction, $E - \sum_{\mu \in M_i} E_\mu \in \ker \psi$ and
divides $u^d(z_i - \beta_i) \prod_{j = 1}^n z_j^{c_j}$. 
Furthermore, we have already bounded $\deg(E)$ in Equation ~\eqref{eqn:e-bound}.
Finally, recall $\ker \psi$ is generated by relations of the form $u^d(z_i - \beta_i) \prod_{j = 1}^n z_j^{c_j}$ as
$i$ ranges between $1$ and $n$. Thus, taking the maximum over all $i$ 
of our bound in Equation ~\eqref{eqn:e-bound}, we see $\ker \psi$ is generated in degrees at most
\[
	\tau = \rho
	+ \max \left(\max_{i\in \Te} \bigl(\ell_i \gcd(a_i, b_i) \bigr),
	\; \max_{\substack{i \in \Tp \\ j \in \Tm}} \bigl(\ell_{i, j}
	(\bida_i \bidb_j - \bida_j
	\bidb_i) \bigr) \right)\le 2\rho.
	\qedhere
\]
\end{proof}

By combining the above results, we get our main theorem bounding
the generator and relation degrees of the section ring of
$\bq$-divisors on Hirzebruch surfaces.

\begin{thm}
\label{thm:hirz-generators-relations-full}
Let $D = \sum_{i = 1}^n \alpha_i D_i \in \di \hirz_m \otimes_\bz \bq$
where $\alpha_i = \frac{c_i}{k_i} \in \bq$ is written in reduced form.
Write $\ell_{i,j} := \lcm_{h\ne i,j}(k_h)$.
Let $u, v, z, w$ be the coordinates for the Hirzebruch surface
$\hirz_m$, as described at the beginning of Section
~\ref{sec:hirz} and suppose that $\{f_1, \ldots, f_n\}$
contains two independent 
linear polynomials in $u, v$ and two independent linear polynomials in
$w, x$ {\rm(}with corresponding $\alpha_j$ possibly zero{\rm)}.
Recall $\Te, \Tp$, and $\Tm$ as given in Definition ~\ref{defn:t-defn} and let each $D_i = V(f_i)$ where $f_i \in \sco(a_i, b_i)$.

Then $R_D$ is generated in degrees at most
\[
	\rho = \sum_{i\in \Te(D)} \gcd(\bida_i, \bidb_i)\ell_i +
	\sum_{\substack{i \in \Tp(D) \\	j \in \Tm(D)}} (\bida_i \bidb_j
	- \bida_j \bidb_i) \ell_{i, j}
\]

\noindent
with relations generated in degrees at most $2 \rho$.
\end{thm}
\begin{proof}
The generation degree bound is as stated in Lemma \ref{lem:hirz-generators}.
By Proposition \ref{lem:composite-map}, the degree of generation of
$\ker \phi$ is at most the maximum of the generation degrees of $\ker \chi$
and $\ker \psi$, giving us the desired relations bound. The bound on $\ker \chi$
follows from Lemma ~\ref{lem:bound-ker-chi} and the bound on $\ker \psi$
follows from Lemma ~\ref{lem:hirz-bound-ker-psi}.
\end{proof}


\section{Further Questions}
\label{sec:conc}
\invisiblesubsection
Recall that every minimal rational surface is either isomorphic to
$\bp^2$ or $\hirz_m$ for some $m \geq 0, m \neq 1$ \cite{eisenbud-harris:minimal}. By Theorems
~\ref{thm:proj-generators-relations} and ~\ref{thm:hirz-generators-relations}, we have given bounds for the generators and relations
of arbitrary section rings on any minimal rational surface. A
natural extension of our results is the following.
\begin{question}
\label{qn:general-minimal-surface}
Can we describe generators and relations of $R_D$ for a divisor $D$ on an
arbitrary rational surface $X$?
\end{question}

Every rational surface can be obtained from a minimal rational surface by a 
sequence of blow-ups \cite{eisenbud-harris:minimal}. Therefore, to answer
Question \ref{qn:general-minimal-surface} affirmatively, it suffices
to bound the degree of generators and relations of the 
section ring of a divisor on a blow-up of a given surface
in terms of the section 
rings of some associated divisors on that given surface.

Another direction to generalize the work in this paper
would be to try to express section rings of $\bq$-divisors on
$X \times Y$ in terms of those on $X$ and $Y$. In this paper,
we bounded the degrees of presentations on section rings on $\bp^1 \times \bp^1 \cong \hirz_0$.
Perhaps similar techniques can be used to bound
degrees of presentations on section rings on $(\bp^1)^k$ or
more generally on $(\bp^1)^{i_1} \times \cdots (\bp^k)^{i_k}.$
One might further try to generalize this to bounding degrees of presentations on bundles over $\bp^m$ or on more general
products of schemes.


\section{Acknowledgments}
\label{sec:ack}
We are grateful to David Zureick-Brown for offering great advice on
this project, for many useful suggestions on this paper, and for
many valuable conversations. We thank Peter
Landesman and Shou-Wu Zhang for offering helpful comments and
suggestions for this paper. We also thank Ken Ono and the
Emory University Number Theory REU for arranging our project and
for the great support that we have received. Finally, we would like to acknowledge the
support of the National Science Foundation (grant number DMS-1250467).


\nocite{*}
\bibliography{bibliography-stacky-surface}
\bibliographystyle{alpha}

\end{document}